\documentclass[article]{siamart0516}

\usepackage[utf8]{inputenc}
\usepackage{amssymb}
\usepackage{amsmath}
\usepackage{graphicx,verbatim,color,booktabs,mathtools}
\usepackage{float}
\usepackage{caption}
\usepackage{subcaption}

\usepackage[normalem]{ulem}

\usepackage{array}
\usepackage{url}
\usepackage{xcolor}
\usepackage{pgfplots}
\usepackage{tikz}
\usetikzlibrary{calc,decorations}

\newcommand{\vol}[0]{\operatorname{vol}}
\newcommand{\norm}[1]{\lVert#1\rVert}
\newcommand{\R}[0]{\mathbb{R}}
\newcommand{\Z}[0]{\mathbb{Z}}
\newcommand{\mycomment}[1]{}

\theoremstyle{plain} 
\newtheorem{defn}{Definition}
\newtheorem{thm}{Theorem}
\newtheorem{prop}{Proposition}
\newtheorem{lem}{Lemma}
\newtheorem{cor}{Corollary}
\theoremstyle{remark}
\newtheorem{rem}{Remark}
\newtheorem{expl}{Example}

\newcommand{\shrinkmargins}[1]{
  \addtolength{\textheight}{#1\topmargin}
  \addtolength{\textheight}{#1\topmargin}
  \addtolength{\textwidth}{#1\oddsidemargin}
  \addtolength{\textwidth}{#1\evensidemargin}
  \addtolength{\topmargin}{-#1\topmargin}
  \addtolength{\oddsidemargin}{-#1\oddsidemargin}
 \addtolength{\evensidemargin}{-#1\evensidemargin}
  }

\shrinkmargins{.7}

\newif\ifcomment
\commenttrue
\newcommand{\ch}[1]{\ifcomment \textcolor{magenta}{Camilla: #1} \fi}

\title{A Design Criterion for the Rayleigh Fading Wiretap Channel Based on $\ell^1$-norm Theta Functions}
\author{Niklas Miller\thanks{This work has been supported by the Research Council of Finland (Grant \#351271, PI Camilla Hollanti).}}
\date{\today}

\begin{document}

\maketitle

\begin{abstract}
   We show that the correct decoding probability of an eavesdropper and error probability of a legitimate receiver in a Rayleigh fading wiretap channel with lattice coset coding are both upper bounded by the theta function in the $\ell^1$-norm of the dual code lattices. Motivated by these findings, we derive a closed form expression for the $\ell^1$-norm theta function of any sublattice of $\mathbb{Q}^n$ and its dual, and prove that the lattice $\mathbb{Z}^n$ minimizes the theta function among all scalings of $\mathbb{Z}^n$ along the coordinate axes. Furthermore, we develop a method to algorithmically find locally critical lattice packings of the $n$-dimensional cross-polytope. Using this method, we are able to construct a four-dimensional lattice having a packing density of $\Delta\approx 0.824858$ and kissing number 30, improving on the best known lattice packing density of the cross-polytope in dimension four.

\end{abstract}

\begin{keywords}$\ell^1$-norm theta functions, Cross-polytope packings, Wiretap codes, Wireless communications.\end{keywords}

\section{Introduction}

Recently, Belfiore and Oggier \cite{belfiore2011lattice} analyzed the correct decoding probability of an eavesdropper in a Rayleigh fading wiretap channel with lattice coset coding. They further derived an approximation for the probability, called the \emph{inverse norm sum} approximation. From this, various heuristics were developed relating different parameters of the code lattices to the probability estimate. However, there is still a lack of an explicit answer to the question of which lattices, or classes of lattices, optimize these expressions. It is therefore the purpose of this paper to show that both the correct decoding probability of the eavesdropper, and error probability of the legitimate receiver, are upper bounded by the theta function in the $\ell^1$-norm of the dual code lattices. This new design criterion, together with the fact that the $\ell^1$-norm theta function of a lattice is dictated by the cross-polytope packing density for small $q$, leads to the search for dense lattice packings of the cross-polytope. For this, a program using nonlinear optimization is developed in Section \ref{sec:cross-polytope-program}.

Sol\'{e} \cite{sole1995counting} studied the $\ell^1$-theta function (which he calls the $\nu$-function), and derived a closed form expression for the theta function of lattices constructed from binary codes via Construction A. We extend his results by deriving a formula for the theta function of any sublattice $\Lambda$ of $\Z^n$, and its dual. We also prove an $\ell^1$-analogue of Montgomery's result \cite[Theorem 2]{montgomery1988minimal} that the theta function of $\Z^n$ is smaller than the theta function of $\alpha_1\Z\oplus\cdots\oplus \alpha_n\Z$, where $\alpha_1\cdots \alpha_n=1$, for all $q$.

\subsection{Number- and coding-theoretic tools}

In this article, a lattice $\Lambda$ is the $\Z$-span of a basis $\{b_1,\dots, b_n\}$ of $\R^n$, i.e. $\Lambda=\{\sum_{i=1}^n c_i b_i:c_i\in \Z\}$. The matrix $B$ with rows $b_1,\dots, b_n$ is the generator matrix of $\Lambda$. The volume of any fundamental domain of $\R^n$ with respect to the translation action of $\Lambda$ is denoted by $\text{vol}(\Lambda)$, and can be computed by $\text{vol}(\Lambda)=|\det(B)|$ for any generator matrix $B$ of $\Lambda$. The dual lattice of $\Lambda$ is $\Lambda^*=\{y\in \R^n:x\cdot y\in\Z\text{ for all $x\in \Lambda$}\}$ and it is generated by $(B^{-1})^{T}$ where $B$ is any generator matrix of $\Lambda$. Let $p\geq 1$ and let $\norm{x}_p:=(\sum_{i=1}^n |x_i|^{p})^{\frac{1}{p}}$ denote the usual $p$-norm in $\R^n$. Define the theta function in the $p$-norm of $\Lambda$ as the generating function
$$\Theta_{\Lambda}^{p}(q)=\sum_{x\in \Lambda} q^{\norm{x}_p^{p}}.$$
$\Theta_{\Lambda}^2(q)$ is the classical theta function, for which the substitution $q=e^{\pi i z}$ makes the theta function a holomorphic function in the upper half-plane $\mathcal{H}=\{z\in\mathbb{C}:\Im(z)>0\}$. Note that the series $\Theta_{\Lambda}^p(q)$ converges for all real $0\leq q<1$, due to the equivalence of norms in $\R^n$ and since $\Theta_{\Lambda}^2(q)$ converges for $0\leq q<1$. We define the Fourier-transform of an $L^1(\R^n)$ function $f:\R^n\to \R$ by (note the sign convention) $$\widehat{f}(\xi)=\int_{\R^n} f(x) e^{2 \pi i x\cdot \xi}d x.$$

\begin{thm}{(Poisson summation formula)}
\label{thm:Poisson}
    Let $f:\R^n\to\mathbb{C}$ be a Schwartz function, and $\Lambda\subseteq\R^n$ a lattice. Then
    $$\sum_{x\in \Lambda} f(x)=\frac{1}{\vol{\Lambda}}\sum_{y\in\Lambda^{*}}\widehat{f}(y).$$
\end{thm}

Recall that a Schwartz function is a $C^\infty(\R^n,\mathbb{C})$ function which decays faster in absolute value than any inverse polynomial as $|x|\to\infty$, and the same is true for all its partial derivatives. Also, the Fourier-transform is an automorphism on the vector space of Schwartz functions.

A linear $q$-ary ($q\geq 2$) code $C$ of length $n$ is a $\Z$-submodule of $(\Z/q\Z)^n$. The dual of $C$ is defined as $C^\perp=\{y\in (\Z/q\Z)^n:y\cdot x\equiv0\pmod{q}\text{ for all $x\in C$} \}$. Weight enumerators capture the distribution of weights in a code. Two important such weight enumerators are the Hamming and symmetric weight enumerators. 

\begin{defn}
    Let $C$ be a linear $q$-ary code of length $n$. The symmetric weight enumerator of $C$ is defined as the homogeneous polynomial of degree $n$ in $\lfloor\frac{q}{2}\rfloor+1$ variables given by
    $$swe_C(x_0,x_1,\dots,x_{\lfloor\frac{q}{2}\rfloor})=\sum_{c\in C}x_0^{n_0(c)}x_1^{n_1(c)}\cdots x_{\lfloor\frac{q}{2}\rfloor}^{n_{\lfloor\frac{q}{2}\rfloor}(c)},$$
    where $n_i(c)$ is the number of coordinates of $c$ which are congruent to $\pm i\pmod{q}$. The Hamming weight enumerator of $C$ is the homogeneous polynomial of degree $n$ in two variables defined as
    $$W_C(x,y)=\sum_{c\in C} x^{n-|c|}y^{|c|},$$
    where $|c|=|\{i:c_i\neq 0\}|$ is the Hamming weight of $c\in C$.
\end{defn}

For $q=2,3$, the symmetric weight enumerator reduces to the Hamming weight enumerator. The weight enumerators of a code and its dual are strongly linked via MacWilliams' identities. We will need the following general variant of MacWilliams' identities, due to Zhu \cite{zhu2003symmetrized}.

\begin{lem}\label{lem:Shixin}
    Let $C$ be a linear $q$-ary code of length $n$. Let $t=\lfloor\frac{q}{2}\rfloor$. Then
    $${swe}_{C^\perp}(x_0,\dots, x_t)=\frac{1}{|C|}{swe}_{C}(y_0,\dots, y_t),$$
    where
    \begin{align*}
        y_i=\begin{cases}
            x_0+2\sum_{j=1}^{t-1}\cos\left(\frac{2\pi ij}{q}\right)x_j+\cos(\pi i)x_t,&\text{if $q$ is even,}\\
            x_0+2\sum_{j=1}^t\cos\left(\frac{2\pi ij}{q}\right)x_j,&\text{if $q$ is odd,}
        \end{cases}
    \end{align*}
    for all $i=0,\dots, t$.
\end{lem}

There is a one-to-one correspondence between sublattices of $\Z^n$ containing $q\Z^n$ and linear $q$-ary codes via Construction A, which associates to a lattice $\Lambda$ satisfying $q\Z^n\subseteq \Lambda\subseteq \Z^n$ the code $C=\Lambda/q\Z^n$. The lattice obtained from a linear $q$-ary code $C$ via Construction A will be denoted $\Lambda_{A(C)}=\{x\in \Z^n:x\equiv c\pmod{q}\in C\}$. Any sublattice $\Lambda\subseteq \Z^n$ is a $q$-ary lattice for $q=\text{vol}(\Lambda)$. In fact, if $R$ is a principal ideal domain and $M_1\subseteq M_2$ are two free $R$-modules of rank $n$, then we can find bases $b_1,\dots,b_n$ of $M_1$ and $b_1',\dots, b_n'$ of $M_2$ such that $b_i=\alpha_i b_i'$, where $\alpha_i\in R$ $(i=1,\dots,n)$ are the invariant factors of the quotient module $M_2/M_1$. They satisfy $\alpha_i\mid \alpha_{i+1}$ for $i=1,2,\dots,n-1$. In the case that we are interested in, $\Lambda\subseteq \Z^n$, we may write $\Lambda=\bigoplus_{i=1}^n \alpha_i b_i'\Z$ for some basis $\{b_1',\dots,b_n'\}$ of $\Z^n$. It follows that $\Lambda$ is an $\alpha_n$-ary lattice. The invariant factors can in practice be found by computing the Smith normal form of a generator matrix of $\Lambda$. Everything here applies to rational lattices $\Lambda\subseteq \mathbb{Q}^n$ as well, since after scaling they are sublattices of $\Z^n$.

\subsection{Channel model}
\label{sec:channel_model}

A wiretap channel, introduced by Wyner \cite{wyner1975wire}, is a channel model where two parties, A and B, communicate over a discrete and memoryless noisy channel under the presence of an eavesdropper, E. We consider a wireless transmission system, in which there is perfect channel state information (CSI) at the receiver. This allows one to consider a real channel model as follows. The transmitted symbols $x\in\R^n$ are codewords in a finite subset $\mathcal{S}$ of $\R^n$, and the received vector is given by
\begin{align*}
    r &= \text{diag}(\alpha_1,\dots,\alpha_n) x + n,
\end{align*}
where $n$ is a normal random variable with mean $0$ and variance $\sigma_b^2$ (intended receiver, B) or $\sigma_e^2$ (unwanted receiver, E). It is assumed that the channel quality of B is superior to that of E, i.e. $\sigma_b^2<\sigma_e^2$. In the additive white Gaussian noise channel (AWGN), $\alpha_i=1$ for all $i$, while in the Rayleigh fading channel, the $\alpha_i$ are independent Rayleigh distributed variables with parameter $\sigma_{h,b}$ and $\sigma_{h,e}$, respectively. The independence assumption relies on the fact that the sent symbols are interleaved before entering the channel. Define $\gamma_e:=\frac{\sigma_{h,e}^2}{\sigma_e^2}$, and similarly $\gamma_b:=\frac{\sigma_{h,b}^2}{\sigma_b^2}$; the average SNR of E respectively B.

We consider lattice coset coding as in \cite{ozarow1984wire}, \cite{wyner1975wire}, where the signal set is a finite subset $\mathcal{S}$ of a lattice $\Lambda_b$ and we choose a sublattice $\Lambda_e\subseteq\Lambda_b$ and identify the set of messages with the elements of $\Lambda_b/\Lambda_e$: the number of possible messages is $\frac{\vol{\Lambda_e}}{\vol{\Lambda_b}}$. The sender chooses a coset representative uniformly at random. The objective is to choose the pair $(\Lambda_b,\Lambda_e)$ in such a way that many different objectives are simultaneously achieved: the information rate is maximized, while the mutual information between the received signal of E and the sent message is minimized, and the correct decoding probability of E is minimized. In this article, we consider the correct decoding probabilities of E and B. Let ECDP denote the correct decoding probability of E.

For the AWGN-channel with lattice coset coding, the ECDP is upper bounded by a constant times the theta function $\Theta_{\Lambda_e}^2\left(\exp\left(-\frac{1}{2\sigma_e^2}\right)\right)$  (in the $\ell^2$-norm) of $\Lambda_e$, as was proved in \cite{oggier2015lattice}.  Furthermore, it is shown in \cite{ling2014semantically} that the mutual information $I(M,R_e)=H(M)-H(M\mid R_e)$, where $R_e$ and $M$ denote the random variables representing the signal received by E and the message sent by A, respectively, and $H(\cdot)$ denotes entropy, is upper bounded by 
$$I(M,R_e)\leq 8\varepsilon_n(n \rho-\log(8\epsilon_n)),$$
where $\rho$ is the total rate of the code and $\epsilon_n=\frac{\vol{\Lambda_e}}{\left(2\pi \sigma_e^2\right)^{\frac{n}{2}}}\Theta_{\Lambda_e}^2\left(\exp(-\frac{1}{2\sigma_e^2})\right)-1$, sometimes called the flatness-factor of $\Lambda_e$. Therefore, both the minimization of ECDP and mutual information lead to the same objective of minimizing the theta function of $\Lambda_e$.

For fading channels (with arbitrarily distributed fading coefficients), one can compute the ECDP by averaging over all possible realizations of the fading. For the specific case of
the Rayleigh fading channel with lattice coset coding, Belfiore and Oggier \cite{belfiore2011lattice} derived the following expression for ECDP:

\begin{equation}
\label{eq:CDP}
    P_{c,e}\approx \left(\frac{\gamma_e}{4}\right)^{\frac{n}{2}}\vol(\Lambda_b)\sum_{x\in \Lambda_e}\prod_{i=1}^n \frac{1}{(1+\gamma_e x_i^2)^{\frac{3}{2}}}.
\end{equation}

Furthermore, they made the assumption that $\gamma_e$ is sufficiently large so that (\ref{eq:CDP}) is approximately

\begin{align}\label{eq:inverse_norm_sum}
       P_{c,e}\approx\left(\frac{1}{4\gamma_e^2}\right)^{\frac{n}{2}}\vol(\Lambda_b)\sum_{x\in\Lambda_e\setminus\{0\}}\prod_{i=1}^n\frac{1}{|x_i|^3},
\end{align}
assuming $\Lambda_e$ has full diversity, i.e. does not have lattice points $x\in \Lambda_e\setminus\{0\}$ with $x_i=0$ for some $i=1,\dots,n$. The above is called the "inverse norm sum" approximation. The above approximation tends to infinity as $\gamma_e\to 0$, so it cannot properly describe the decoding probability as the channel quality becomes poor. Note also the lack of convergence of the above sum: e.g. if $\Lambda_e$ is a lattice constructed via the canonical embedding from the ring of integers $\mathcal{O}_K$ of a totally real algebraic number field $K\neq \mathbb{Q}$, then the sum is $\sum_{x\in\mathcal{O}_K\setminus\{0\}} \frac{1}{|N_{K/\mathbb{Q}}(x)|^3}$, where $N_{K/\mathbb{Q}}$ is the field norm, which diverges since there are infinitely many units in $\mathcal{O}_K$. Of course the exact sums are finite and therefore convergent. Let us mention that the mutual information of the signal received by E and the message sent by A have been analyzed in fading channels e.g. in \cite{mirghasemi2015lattice}, \cite{damir2021WR} and \cite{luzzi2016almost}.

Let us now consider the legitimate receiver's decoding \emph{error} probability in a Rayleigh fading channel with lattice coding. The following is a classic upper bound, derived in \cite{boutros1996good}:

\begin{equation}
    \label{eq:Bob_correct}
    P_{e,b}\leq \frac{1}{2} \sum_{x\in \Lambda_b\setminus\{0\}}\prod_{i=1}^n \frac{1}{1+\frac{\gamma_b}{4} x_i^2}.
\end{equation}

Above the letter $e$ stands for the word "error". Note that the above bound is also relevant in the situation of lattice coset coding: (\ref{eq:Bob_correct}) describes the probability that the received signal $r=\text{diag}(\alpha_1,\dots,\alpha_n)x+n$ of B lies closer to some lattice point $y\in\Lambda_b$ different from $x$. Now, in the situation of lattice coset coding, the point $y\in\Lambda_b$ may lie in the same coset as $x$ modulo $\Lambda_e$, whence there is no decoding error, but this probability is negligible if the channel quality of B is good, and the index $[\Lambda_b:\Lambda_e]$ is large enough so that the closest lattice points to $x\in\Lambda_b$ in $\Lambda_b$ all lie in different cosets modulo $\Lambda_e$.

If B has a good channel quality ($\gamma_b\to \infty$), (\ref{eq:Bob_correct}) is approximated by a factor times a sum of the form $\sum_{x\in\Lambda_b\setminus\{0\}}\prod_{i=1}^n \frac{1}{x_i^2}$, which leads to the heuristic of choosing $\Lambda_b$ to have a large minimum product distance $d(x)=\prod_{i=1}^n |x_i|$ and large diversity, defined as the minimum number of non-zero coordinates in any non-zero lattice vector. We will show in section \ref{sec:probability_estimates} that (\ref{eq:Bob_correct}) can be expressed exactly in terms of the $\ell^1$-norm theta function of the dual lattice $\Lambda_b^*$, which leads to the objective of minimizing the $\ell^1$-norm theta function of $\Lambda_b^*$. Note that this has the advantage of not having to assume that $\gamma_b$ is large, unlike in the approximation above. 

\section{An upper bound from the $\ell^1$-norm theta function}
\label{sec:probability_estimates}

Consider the ECDP (\ref{eq:CDP}) in a Rayleigh fading channel with lattice coset coding. The minimization of that expression is equivalent to minimizing, for a fixed $\gamma>0$, the function

\begin{equation}
\label{eq:obj}
    F(\Lambda;\gamma) := \sum_{x\in \Lambda} \prod_{i=1}^n \frac{1}{(1+\gamma x_i^2)^{\frac{3}{2}}}
\end{equation}

over the space of lattices $\Lambda\subseteq\R^n$ of fixed volume. Consider the Taylor-series of $x\mapsto (1+\gamma x^2)^{\frac{3}{2}}$ around $x_0=0$: $$(1+\gamma x^2)^{\frac{3}{2}}=1+\frac{3 \gamma  x^2}{2}+\frac{3 \gamma ^2 x^4}{8}+\mathcal{O}(x^6).$$
Bernoulli's inequality gives $(1+\gamma x^2)^{\frac{3}{2}}\geq 1+\frac{3\gamma x^2}{2}$ for all $x\in \R$, and $\gamma>0$, which gives the bound


    


\begin{align}
\label{inequality:FG}
    F(\Lambda;\gamma) &= \sum_{x\in \Lambda} \prod_{i=1}^n \frac{1}{(1+\gamma x_i^2)^{\frac{3}{2}}}\leq \sum_{x\in \Lambda} \prod_{i=1}^n \frac{1}{1+\frac{3}{2}\gamma x_i^2}=:G(\Lambda;\gamma).
\end{align}

Note that $G(\Lambda;\gamma)$ has the same shape as the error probability of B in (\ref{eq:Bob_correct}). Furthermore, the series $G(\Lambda;\gamma)$ converges for any fixed $\gamma>0$ by Proposition \ref{prop:2} below and the convergence of the $\ell^p$-norm theta functions for $0\leq q<1$. In fact, $G(\Lambda;\gamma)$ is asymptotic to $\frac{1}{\vol{\Lambda}}\left(\frac{\pi\sqrt{2}}{\sqrt{3\gamma}}\right)^n$ as $\gamma\to 0$, which gives an upper bound of $\left(\frac{\pi}{\sqrt{6}}\right)^n \frac{\vol{\Lambda_b}}{\vol{\Lambda_e}}\approx 1.28255^n \frac{\vol{\Lambda_b}}{\vol{\Lambda_e}}$ on the probability estimate (\ref{eq:CDP}) in the limit $\gamma\to0$. This differs slightly from the uniform guess probability $\frac{\vol{\Lambda_b}}{\vol{\Lambda_e}}$, but is at least finite, contrary to the inverse norm sum approximation. Next, we want to relate the function $G(\Lambda;\gamma)$ to the $\ell^1$-norm theta function of $\Lambda^*$. This is done by taking the Fourier transform of the function $x\mapsto\frac{1}{1+sx^2}$.

\begin{lem}
\label{lem:1}
    The Fourier transform of $f:\R\to \R, \ f(x)=\frac{1}{1+s x^2}$, where $s>0$ is real, is given by $\widehat{f}(\xi)=\frac{\pi}{\sqrt{s}} e^{-2 \pi \frac{|\xi|}{\sqrt{s}}}$.
\end{lem}
\begin{proof}
The Fourier transform of $g(x)=\pi e^{-2\pi |x|}$ is given by
    \begin{align*}
         \widehat{g}(\xi) &= \int_{-\infty}^\infty g(x) e^{2 \pi i \xi x} d x = \pi \left(\int_{-\infty}^0e^{2\pi x(i\xi+1)}dx+\int_{0}^\infty e^{2\pi x(i\xi-1)}d x\right) \\
             &= \pi \left(\frac{1}{2\pi(1+i\xi)}+\frac{1}{2\pi(1-i\xi)}\right)= \frac{1}{1+\xi^2}.
    \end{align*}
    So by the Fourier inversion formula, the Fourier transform of $h(x)=\frac{1}{1+x^2}$ is $\widehat{h}(\xi)=\pi e^{-2\pi |\xi|}$. Using the scaling property, we obtain
    $$\widehat{f}(\xi)=\widehat{h(\sqrt{s}x)}(\xi)=\frac{1}{\sqrt{s}}\widehat{h}(\frac{\xi}{\sqrt{s}})=\frac{\pi}{\sqrt{s}} e^{-2 \pi \frac{|\xi|}{\sqrt{s}}}.$$
\end{proof}

\begin{lem}
\label{lem:2}
    Define the function $f:\R^n\to \R$, $f(x)=\prod_{i=1}^n \frac{1}{1+s x_i^2}$ where $s>0$ is a real number. Then the Fourier transform of $f$ is given by $$\widehat{f}(\xi)=\left(\frac{\pi}{\sqrt{s}}\right)^n e^{-\frac{2\pi}{\sqrt{s}}\norm{\xi}_1}.$$
\end{lem}
\begin{proof}
By lemma \ref{lem:1},
    \begin{align*}
        \widehat{f}(\xi)  &= \int_{\R^n} f(x)e^{2 \pi i \xi \cdot x}d x = \prod_{i=1}^n\int_{-\infty}^\infty \frac{1}{1+s x_i^2}e^{2 \pi i \xi_i x_i} dx_i \\
        &=\prod_{i=1}^n \frac{\pi}{\sqrt{s}} e^{-2 \pi \frac{|\xi_i|}{\sqrt{s}}} = \left(\frac{\pi}{\sqrt{s}}\right)^n e^{-\frac{2\pi}{\sqrt{s}}\norm{\xi}_1}.
    \end{align*}
\end{proof}
Putting things together, we get

\begin{prop}
    \label{prop:2}
    Let $s>0$ be real and $\Lambda\subseteq\R^n$ a lattice. Then
    \begin{equation}
        \label{eq:Poisson}
        \sum_{x\in\Lambda}\prod_{i=1}^n \frac{1}{1+ sx_i^2}= \frac{1}{\vol(\Lambda)}\left(\frac{\pi}{\sqrt{s}}\right)^{n}\Theta_{\Lambda^*}^1 \left(e^{-\frac{2\pi}{\sqrt{s}}}\right),
    \end{equation}
    where $\Theta_{\Lambda^*}^1$ is the theta function of the dual lattice $\Lambda^*$ in the $\ell^1$-norm.
\end{prop}
\begin{proof}
    Apply Theorem \ref{thm:Poisson} to the function $f(x)=\prod_{i=1}^n \frac{1}{1+s x_i^2}$ and lattice $\Lambda$, and use Lemma \ref{lem:2}.
\end{proof}

As a corollary, we obtain the following probability estimates.

\begin{cor}
    \label{cor:probability_bounds}
    The decoding error probability of B in a Rayleigh fading wiretap channel with lattice coding or lattice coset coding is upper bounded by
    \begin{align*}
       P_{e,b} &\leq \frac{1}{2}\left(\frac{1}{\vol{\Lambda_b}}\left(\frac{2\pi}{\sqrt{\gamma_b}}\right)^n\Theta^1_{\Lambda_b^*}(e^{-\frac{4\pi}{\sqrt{\gamma_b}}})-1\right).
       \end{align*}
       The ECDP in the Rayleigh fading wiretap channel with lattice coset coding is upper bounded by
       \begin{align*}
       P_{c,e} &\leq \frac{\vol{\Lambda_b}}{\vol{\Lambda_e}}\left(\frac{\pi}{\sqrt{6}}\right)^{n}\Theta_{\Lambda_e^*}^{1}\left(e^{-\frac{2\sqrt{2}\pi}{ \sqrt{3\gamma_e}}}\right).
    \end{align*}
\end{cor}

Note that the first part of the above corollary is a reformulation of (\ref{eq:Bob_correct}) in the language of $\ell^1$-norm theta functions. It gives the design criterion: minimize the $\ell^1$-norm theta function of $\Lambda_b^*$, at the value $q=e^{-\frac{4\pi}{\sqrt{\gamma_b}}}$. The second part gives the following design criterion for the Rayleigh fading wiretap channel with lattice coset coding: minimize the $\ell^1$-norm theta function of the dual lattice $\Lambda_e^*$, at the value $q=e^{-\frac{2\sqrt{2}\pi}{\sqrt{3 \gamma_e}}}$. This suggest choosing the lattice pair $(\Lambda_b,\Lambda_e)$ on the dual side: first, one selects a lattice $\Lambda_e^*$ with small $\ell^1$-norm theta function, and then a sublattice $\Lambda_b^*\subseteq\Lambda_e^*$ of the desired index with the same property. Note the similarity with the AWGN channel design criterion: there, we want to minimize the $\ell^2$-norm theta function $\Theta_{\Lambda_e}^2\left(e^{-\frac{1}{2\sigma_e^2}}\right)$.

Now, writing $\Theta_{\Lambda}^1(q)=\sum_{k\geq 0} N_k q^k$ where $N_k=|\{x\in \Lambda:\norm{x}_1=k\}|$, we see that the problem of minimizing the $\Theta_{\Lambda}^1$-function and minimizing the shortest lattice vector length in the $\ell^1$-norm become equivalent in the limit as $q\to 0$ i.e. $\gamma_e\to 0$. Although the $\Theta^1$-function lacks some of the properties that the usual theta function has, like modularity and invariance under orthogonal transformations, a lot can be said about the theta function of rational lattices and their duals, see Section \ref{sec:theta_functions}.

In the $\ell^2$-norm, the minimization of the theta function in specific dimensions is an important and well-studied problem. A  result by Montgomery \cite{montgomery1988minimal} states that the densest lattice packing of balls in dimension 2, the hexagonal packing, also minimizes the theta function for $all$ real values $0\leq q<1$. In Remark \ref{rem:theta_minimum} it is proved that the same is false for the densest lattice packing of the $\ell^1$-ball in dimension 2. Moreover, in dimensions 8 and 24, respectively, the densest sphere packings, the $E_8$ and Leech lattice both minimize the theta function for all real values of $q$ \cite{cohn2022universal}. In dimension 3, it is widely conjectured that  the $D_3$ lattice minimizes the theta function for all $q\leq e^{-\pi}$, while its dual $D_3^*$ is the minimizer for $q\geq e^{-\pi}$, and equality holds at $q=e^{-\pi}$ by Poisson summation. Note that in order to compare the theta functions of different lattices, they must be scaled to have the same volume.

\section{$\Theta^1$-functions of lattices}
\label{sec:theta_functions}

In this section, we derive a closed form expression for the $\Theta^1$-function of any lattice $\Lambda\subseteq\Z^n$ in terms of the symmetric weight distribution of $\Lambda/m\Z^n$, where $m$ is the smallest positive integer such that $m\Z^n\subseteq \Lambda$. It was proved in \cite[Theorem 3]{sole1995counting} that $\Theta_{\Lambda}^1(q)$ is a rational function of $q$; the next theorem gives a description of what that rational function is.

\begin{thm}
\label{thm:Theta_l1_norm}
    Let $C$ be an $m$-ary linear code of length $n$ and $\Lambda_{A(C)}$ the associated Construction A lattice. Then the $\Theta^1$-function of $\Lambda_{A(C)}$ is given by
    $$\Theta^1_{\Lambda_{A(C)}}(q)=\frac{1}{(1-q^m)^n} swe_C(1+q^m,q+q^{m-1},\dots, q^{\lfloor\frac{m}{2}\rfloor}+q^{m-\lfloor\frac{m}{2}\rfloor}).$$
\end{thm}

\begin{proof}
    Let $x\in[0,1)$. Then
    \begin{align*}
        \Theta^1_{\Z+x}(q) &= \sum_{z\in\Z}q^{|z+x|}=\sum_{z\geq -x} q^{z+x}+\sum_{z<-x}q^{-z-x}\\
        &=q^x\sum_{z=0}^{\infty}q^z +q^{-x}\sum_{z=1}^\infty q^z=\frac{1}{1-q}(q^x+q^{1-x}).
    \end{align*}
    Thus, partitioning $\Lambda_{A(C)}$ as $\Lambda_{A(C)}=\cup_{c\in C}(m\Z^n+c)$ and using the multiplicativity of the theta function we get
    \begin{align*}
        \Theta^1_{\Lambda_{A}(C)}(q) &= \sum_{c\in C} \Theta^1_{m \Z^n+c}(q)=\sum_{c\in C}\prod_{i=1}^n \Theta^1_{m\Z+ c_i}(q)\\
        &=\sum_{c\in C}\prod_{i=1}^n \Theta^1_{\Z+ \frac{c_i}{m}}(q^m)= \frac{1}{(1-q^m)^n} \sum_{c\in C}\prod_{i=1}^n(q^{c_i}+q^{m-c_i})\\
        &=\frac{1}{(1-q^m)^n} swe_C(1+q^m,q+q^{m-1},\dots, q^{\lfloor\frac{m}{2}\rfloor}+q^{m-\lfloor\frac{m}{2}\rfloor}).
    \end{align*}
\end{proof}

As a corollary we obtain the following special cases. The binary case was proven by Sol\'{e} \cite{sole1995counting}.

\begin{cor}
\label{cor:theta_binary_ternary}
    If $C$ is a binary linear code of length $n$ with Hamming weight enumerator $W_C(x,y)$, then
    $$\Theta^1_{\Lambda_A(C)}(q)=\frac{W_C(1+q^2,2q)}{(1-q^2)^n}.$$
    If $C$ is a ternary linear code of length $n$ with Hamming weight enumerator $W_C(x,y)$, then
    $$\Theta^1_{\Lambda_A(C)}(q)=\frac{W_C(1+q^3,q+q^2)}{(1-q^3)^n}.$$
\end{cor}

Using the general variant of MacWilliams' identity, Lemma \ref{lem:Shixin}, one can obtain a closed form expression for the $\Theta^1$-function of the dual lattice $\Lambda_{A(C)}^*=\frac{1}{q}\Lambda_{A(C^\perp)}$ as well. 

\begin{prop}
    \label{prop:Theta_l1_dual}
    Let $C$ be a linear $m$-ary code of length $n$ and let $C^\perp$ denote its dual. Then
    \begin{align*}
        \Theta_{\Lambda_{A(C^\perp)}}^1(q)=\frac{(1-q^2)^n}{|C|}swe_C(y_0,y_1,\dots,y_{\lfloor\frac{m}{2}\rfloor}),
    \end{align*}
    where $$y_k=\frac{1}{q^2-2q\cos\left(\frac{2 \pi k}{m}\right)+1}\quad\text{for all $k=0,1,\dots,\lfloor\frac{m}{2}\rfloor$}.$$
\end{prop}

\begin{proof}
Let $t=\lfloor\frac{m}{2}\rfloor$, and let $\zeta_m:=e^{\frac{2 \pi i}{m}}$. By Theorem \ref{thm:Theta_l1_norm} and Lemma \ref{lem:Shixin},
\begin{align*}
    \Theta_{\Lambda_{A(C^\perp)}}^1(q)&=\frac{1}{(1-q^m)^n} swe_{C^{\perp}}(1+q^m,q+q^{m-1},\dots, q^{\lfloor\frac{m}{2}\rfloor}+q^{m-\lfloor\frac{m}{2}\rfloor}) \\
    &=\frac{1}{|C|(1-q^m)^n}swe_{C}(y_0,y_1,\dots, y_t),
\end{align*}
where
\begin{align*}
        y_k=\begin{cases}
            1+q^m+\sum_{j=1}^{t-1}(\zeta_m^{kj}+\zeta_{m}^{-kj})(q^j+q^{m-j})+2\cos(\pi k) q^{\frac{m}{2}},&\text{if $m$ is even,}\\
            1+q^m+\sum_{j=1}^t(\zeta_m^{kj}+\zeta_{m}^{-kj})(q^j+q^{m-j}),&\text{if $m$ is odd.}
        \end{cases}
    \end{align*}
Above we used the fact that $\cos\left(\frac{2\pi x}{m}\right)=\frac{\zeta_m^x+\zeta_m^{-x}}{2}$. Consider the case $m$ odd first. Then
{\small %
\begin{align*}
    &\sum_{j=1}^t q^j(\zeta_m^{kj}+\zeta_m^{-kj}) = q \zeta_m^k\frac{q^t \zeta_m^{tk}-1}{q\zeta_m^k-1}+q \zeta_m^{-k}\frac{q^t \zeta_m^{-tk}-1}{q\zeta_m^{-k}-1} \\
    &=\frac{q}{q^2-2 \cos\left(\frac{2\pi k}{m}\right)q+1}((q-\zeta_m^k)(q^t \zeta_m^{tk}-1)+(q-\zeta_m^{-k})(q^t \zeta_m^{-t k}-1))\\
    &=\frac{2q}{q^2-2 \cos\left(\frac{2\pi k}{m}\right)q+1}\left(\cos\left(\frac{2\pi t k}{m}\right)q^{t+1}-\cos\left(\frac{2\pi(t+1)k}{m}\right)q^t-q+\cos\left(\frac{2\pi k}{m}\right)\right).
\end{align*}
} %

Similarly, using that $t=\frac{m-1}{2}$ and replacing $q$ by $q^{-1}$ above and multiplying by $q^m$,
{\footnotesize %
\begin{align*}
    &\sum_{j=1}^t q^{m-j} (\zeta_m^{kj}+\zeta_m^{-kj}) \\
    &= \frac{2 q^{m+1}}{q^2-2\cos\left(\frac{2\pi k}{m}\right)q+1}\left(\cos\left(\frac{2\pi t k}{m}\right)q^{-(t+1)}-\cos\left(\frac{2\pi(t+1)k}{m}\right)q^{-t}-q^{-1}+\cos\left(\frac{2\pi k}{m}\right)\right) \\
    &= \frac{2 q}{q^2-2\cos\left(\frac{2\pi k}{m}\right)q+1}\left(\cos\left(\frac{2\pi t k}{m}\right)q^{t}-\cos\left(\frac{2\pi(t+1)k}{m}\right)q^{t+1}-q^{m-1}+\cos\left(\frac{2\pi k}{m}\right)q^m\right).
\end{align*}
} %
Note that 

\begin{align*}
    \cos\left(\frac{2\pi t k}{m}\right)-\cos\left(\frac{2\pi (t+1) k}{m}\right)&=\cos\left(\frac{\pi (m-1) k}{m}\right)-\cos\left(\frac{\pi (m+1)k}{m}\right)\\
    &=\cos\left(\pi k-\frac{\pi k}{m}\right)-\cos\left(\pi k+\frac{\pi k}{m}\right)=0
\end{align*} 

for all integral values of $k$. Therefore,
\begin{align*}
  &1+q^m+\sum_{j=1}^t (q^j+q^{m-j})(\zeta_m^{kj}+\zeta_m^{-kj}) 
  \\
  &=1+q^m+ \frac{2 q}{q^2-2\cos\left(\frac{2\pi k}{m}\right)q+1}\left(\cos\left(\frac{2\pi k}{m}\right)(q^m+1)-(q^{m-1}+q)\right)\\
  &= \frac{(1-q^2)(1-q^m)}{q^2-2\cos\left(\frac{2\pi k}{m}\right) q+1}.
\end{align*}
The case $m$ even is analogous and thus omitted. The claim follows from the fact that $swe_C$ is a homogeneous polynomial of degree $n$.
\end{proof}

We can extract the following special cases from the previous proposition:

\begin{cor}
    \label{cor:theta_binary_ternary_duals}
    If $C$ is a binary linear code of length $n$ with Hamming weight enumerator $W_C(x,y)$, then
    $$\Theta^1_{\Lambda_A(C^\perp)}(q)=\frac{1}{|C|}W_C\left(\frac{1+q}{1-q},\frac{1-q}{1+q}\right).$$
    If $C$ is a ternary linear code of length $n$ with Hamming weight enumerator $W_C(x,y)$, then
    $$\Theta^1_{\Lambda_A(C^\perp)}(q)=\frac{1}{|C|}W_C\left(\frac{1+q}{1-q},\frac{1-q^2}{1+q+q^2}\right).$$
\end{cor}

Below are some examples. Note that the $\ell^1$-norm theta function is not invariant under all rotations of the lattice. However, the group of signed permutations preserves the $\ell^1$-norm.
\begin{expl}
    The lattice $\Z^n$ is obtained from the binary code $C=\mathbb{F}_2^n$, with weight enumerator $W_C(x,y)=(x+y)^n$. Its $\Theta^1$-function is given by $$\Theta^1_{\Z^n}(q)=\left(\frac{1+q}{1-q}\right)^n=1 + 2 n q + 2 n^2 q^2 + \frac{2n}{3}(1+2n^2)q^3+\frac{2n^2}{3}(2+n^2)q^4\dots$$
\end{expl}

\begin{expl}
    The lattice $D_n$ is obtained from the binary code $C=\{c\in \mathbb{F}_2^n: \sum_{i=1}^n c_i\equiv 0\pmod{2}\}$, with weight enumerator $W_C(x,y)=\frac{1}{2}((x+y)^n+(x-y)^n)$. Its $\Theta^1$-function is given by $$\Theta^1_{D_n}(q)=\frac{(1+q)^{2n}+(1-q)^{2n}}{2(1-q^2)^n}=1 + 2 n^2 q^2 +\frac{2n^2}{3}(2+n^2)q^4\dots$$
\end{expl}

\begin{expl}
  The dual lattice of $D_n$ is $D_n^*=\frac{1}{2}\Lambda_{A(C^\perp)}$ where $C$ is the above even weight code, and $C^\perp$ is the "repetition code" of length $n$ and size $2$. Corollary \ref{cor:theta_binary_ternary_duals} gives
  $$\Theta_{D_n^*}^1(q)=\Theta_{\Lambda_{A(C^\perp)}}^1(q^{\frac{1}{2}})=\frac{1}{(1-q)^n}((1+q)^n+(2 \sqrt{q})^n)=\Theta_{\Z^n}^1(q)+2^n q^{\frac{n}{2}}+n 2^nq^{\frac{n}{2}+1}\dots,$$
  which also follows from Corollary \ref{cor:theta_binary_ternary} and the fact that the repetition code has weight enumerator $W_{C^\perp}(x,y)=x^n+y^n$.
\end{expl}

\begin{expl}
    $E_8$ is the unique even unimodular lattice in dimension $8$, and the densest sphere packing in its dimension. It is equal to $E_8=\frac{1}{\sqrt{2}}\Lambda_{A(\widetilde{H_7})}$ where $\widetilde{H_7}$ is the binary $[n,k,d]=[8,4,4]$ extended Hamming code with weight enumerator $W_{\widetilde{H_7}}(x,y)=x^8+14 x^4 y^4+y^8$. We have $\Theta_{E_8}^1(q)=\Theta_{\Lambda_{A(\widetilde{H_7})}}^1(q^{\frac{1}{\sqrt{2}}})$ where 
    \begin{align*}
        \Theta_{\Lambda_{A(\widetilde{H_7})}}^1(q)&=\frac{1}{(1-q^2)^8}((1+q^2)^8+14(1+q^2)^4(2q)^4+2^8 q^8)\\
        &=1 + 16 q^2 + 352 q^4 + 3376 q^6 + 19648 q^8 +\dots
    \end{align*}
\end{expl}

\begin{expl}
    Consider the quaternary Golay code $\widehat{Q}_4$, see \cite{BonnecazeSoleCalderbank}. This code has symmetric weight enumerator
    \begin{align*}
        swe_{\widehat{Q}_4}(x,y,z) &= x^{24}+759 x^{16} z^8+12144 x^{14} y^8 z^2+170016 x^{12} y^8 z^4+2576 x^{12} z^{12}\\
        &+61824 x^{11} y^{12} z+765072 x^{10} y^8 z^6+1133440 x^9 y^{12} z^3+24288 x^8 y^{16}\\
        &+1214400 x^8 y^8 z^8+759 x^8 z^{16}+4080384 x^7 y^{12} z^5+680064 x^6 y^{16} z^2\\
        &+765072 x^6 y^8 z^{10}+4080384 x^5 y^{12} z^7+1700160 x^4 y^{16} z^4\\
        &+170016 x^4 y^8 z^{12}+1133440 x^3 y^{12} z^9+680064 x^2 y^{16} z^6+12144 x^2 y^8 z^{14}\\
        &+61824 x y^{12} z^{11}+4096 y^{24}+24288 y^{16} z^8+z^{24}.
    \end{align*}
    Applying Construction A to this code produces the Leech lattice (the unique even unimodular 24-dimensional lattice without vectors of length $2$) \cite[Theorem 4.3]{BonnecazeSoleCalderbank}: $\Lambda_{\text{Leech}}=\frac{1}{2}\Lambda_{A(\widehat{Q}_4)}$. The $\ell^1$-norm theta function is $\Theta_{\Lambda_{\text{Leech}}}^1(q)=\Theta_{\Lambda_{A(\widehat{Q}_4)}}^1(q^{\frac{1}{2}})$ where
    $$\Theta_{\Lambda_{A(\widehat{Q}_4)}}^1(q)=1 + 48 q^4 + 1152 q^8 + 67024 q^{12} + 512256 q^{14} + 7850592 q^{16} + 
 61193984 q^{18}+\dots$$
\end{expl}

\subsection{Minimal theta functions}

Montgomery proved in his paper \cite[Theorem 2]{montgomery1988minimal}, that the theta function satisfies the inequality $\Theta^2_{\Z^n}(q)\leq \Theta_{\alpha_1 \Z\oplus \cdots \oplus \alpha_n \Z}^2(q)$ for all $0< q<1$, where
\begin{align}
    \label{eq:assumptions}
    \alpha_1\alpha_2\cdots \alpha_n=1 \text{ and } \alpha_i>0 \text{ for all $i=1,\dots,n$.}
\end{align}
Furthermore, the inequality is strict unless $\alpha_i=1$ for all $i=1,\dots,n$. Hence, the theta function is minimized at the only well-rounded lattice in the class of lattices of the form $\alpha_1\Z\oplus\cdots\oplus\alpha_n \Z$. Let us now prove the analogous result for the $\ell^1$-norm theta function. First, we need an auxiliary lemma.

\begin{lem}
    \label{lem:auxiliary}
    $\sinh(\alpha^2 z)\geq \alpha^2 \sinh(z)$ for all $\alpha\geq 1$, $z>0$ and the inequality is strict unless $\alpha=1$.
\end{lem}
\begin{proof} Fix $z>0$. Equality holds at $\alpha=1$, and for $\alpha\geq 1$,
    \begin{align*}
        \frac{\partial}{\partial \alpha}(\sinh(\alpha^2 z)-\alpha^2 \sinh(z)) &=2\alpha(z\cosh(\alpha^2 z)-\sinh(z)\\
        &\geq 2\alpha(z\cosh(z)-\sinh(z))> 0
    \end{align*}
    since $0\cosh(0)-\sinh(0)=0$ and $\frac{d}{dz}(z\cosh(z)-\sinh(z))=z\sinh(z)>0$.
    \end{proof}
\pagebreak
\begin{prop}
\label{prop:orthogonal_theta_minimization}
    Let $\alpha_1,\dots,\alpha_n$ be real numbers satisfying (\ref{eq:assumptions}). Then $\Theta_{\Z^n}^1(q)\leq \Theta_{\alpha_1 \Z\oplus \cdots \oplus \alpha_n \Z}^1(q)$ for all $0<q<1$, and the inequality is strict unless $\alpha_i=1$ for all $i=1,\dots, n$.
\end{prop}
\begin{proof}
    It is enough to prove that $\Theta_{\Z^2}^1(q)\leq \Theta_{\alpha\Z\oplus\frac{1}{\alpha}\Z}^1(q)$ for all $\alpha\geq 1$ and $0<q<1$, with equality if and only if $\alpha=1$. Making the substitution $q=e^{-2y}$, where $y>0$, gives
    \begin{align*}
        \Theta_{\Z^2}^1(q)=\left(\frac{1+e^{-2y}}{1-e^{-2y}}\right)^2&=\coth(y)^2,\\
        \Theta_{\alpha\Z\oplus\frac{1}{\alpha}\Z}^1(q)&=\coth(\alpha y)\coth\left(\frac{y}{\alpha}\right).
    \end{align*}
    So the claim reduces to proving that $f(\alpha,y)=\frac{\coth(\alpha y)\coth\left(\frac{y}{\alpha}\right)}{\coth(y)^2}\geq 1$ when $\alpha\geq 1$, $y>0$, and equality holds if and only if $\alpha=1$. Let $y>0$ be fixed. A simple computation gives that
    \begin{align*}
        \frac{\partial}{\partial \alpha} f(\alpha,y) &= \frac{y \tanh(y)^2}{ 2\alpha^2 \sinh(\alpha y)^2\sinh\left(\frac{y}{\alpha}\right)^2}\left(\sinh(2\alpha y)-\alpha^2\sinh\left(\frac{2 y}{\alpha}\right)\right).
    \end{align*}
    Let $z=\frac{2y}{\alpha}>0$ and apply the preceding lemma to obtain that $\alpha\mapsto f(\alpha,y)$ is strictly increasing on $\alpha\in [1,\infty)$.
\end{proof}

\begin{rem}
    \label{rem:theta_minimum}
    The minimizer of the $\Theta^1$-function for sufficiently small $q$ is given by the densest cross-polytope packing in the given dimension. However, even in dimension two, the densest cross-polytope packing does not minimize the theta-function for all values of $q$: The lattice $L_1$ generated by $2\Z^2$ and $(1,1)$ produces a lattice packing of cross-polytopes of packing density 1, while the lattice $L_2$ generated by $3\Z^2$ and $(1,1)$ has worse packing density. But $\Theta_{L_2}^1(q)<\Theta_{L_1}^1(q)$ if $q>0.0162678\dots$ (after scaling lattices to unit volume). This phenomenon does not occur with the $\ell^2$-norm theta function in dimensions $2$, $8$ and $24$, as remarked  earlier.
\end{rem}

\section{Dense lattice packings of cross-polytopes}
\label{sec:dense_crosspolytopes}

In this section, motivated by the wiretap setting and minimization of $\ell^1$-norm theta functions in the small $q$-range, we consider the problem of densely lattice packing the $n$-dimensional cross-polytope in $\R^n$. Ideally, dense lattice packings of the cross-polytope should give good lattice codes for the fading wiretap channel, since they have small $\Theta_1$-function for sufficiently small $q$.

Let $K\subseteq\R^n$ be a centrally symmetric convex body. By a packing of $K$ we mean a family $\mathcal{P}=\{K_i\}_{i\in I}$ of translates (i.e., $K_i=K+s_i$ for some $s_i\in \R^n$) of $K$ such that the interiors of $K_i$ and $K_j$ are disjoint for all $i,j\in I$, $i\neq j$. A lattice packing is a packing consisting of the translates of $K$ by vectors in a lattice $\Lambda$ in $\R^n$. Informally, the density $\Delta(\mathcal{P})$ of a packing $\mathcal{P}$ is the proportion of space occupied by the elements in the packing. The packing density $\Delta_n(K)$ is the supremum of the packing densities over all possible packings of $K$ in $\R^n$. 

Let $K_n:=\{x\in \R^n:\sum_{i=1}^n|x_i|\leq 1\}$ denote the $n$-dimensional cross-polytope. Recall that the volume of $K_n$ is given by $\text{vol}(K_n)=\frac{2^n}{n!}$, and given a lattice $\Lambda\subseteq \R^n$, the induced lattice packing of $K_n$ has density
\begin{equation}
\label{eq:lattice_crosspacking_density}
    \Delta=\frac{\lambda_{1}(\Lambda)^n}{n!\text{vol}(\Lambda)},
\end{equation}
where $\lambda_1(\Lambda):=\min\{\norm{x}_1:x\in \Lambda\setminus\{0\}\}$ is the length of the shortest non-zero lattice vector in the $\ell^1$-metric.

The Minkowski--Hlawka theorem implies the existence of lattice packings of $K_n$ of density $\log(\Delta)\geq -\frac{n}{2}+o(n)$ as $n\to \infty$. Constructive lattice packings of $K_n$ having density $\log(\Delta)\geq -n\log\log(n)+O(n)$ were described by Rush \cite{rush1991constructive}. In the other direction, Fejes T\'oth, Fodor and V\'igh \cite{toth2015packing} proved that the packing density satisfies $\Delta_n(K_n)\leq c^n$ for some constant $0\leq c\leq 0.87$ as $n\to \infty$, a bound which follows from an asymptotic bound on the packing density of the $\ell^2$-ball and by inscribing a maximum sized ball inside the cross-polytope. 

In small dimensions, the exact packing density of $K_n$ is only known in dimensions $1,2$ and $4$, where the packing density is $1$. The lattice packing density of $K_3$ was proved to be $\frac{18}{19}$ by Minkowski \cite{minkowski}, but the only non-trivial non-lattice bound on the packing density of $K_3$ is due to Graver, Elser and Kallus \cite{gravel2011upper}: $\Delta_3(K_3)\leq 1-1.4\cdot 10^{-12}$. It should be noted that in dimension 4, the tiling of $\R^4$ by $K_4$ is a non-lattice tiling, achieved by the 16-cell honeycomb. The lattice packing density of $K_4$ is unknown: previously, the densest known four-dimensional lattice packing of the cross-polytope was due to Cools and Lyness \cite{cools2001three}. The lattice they found has a packing density of $\Delta = \frac{512}{621}\approx 0.824477$. The lattice found by our methods, described in Section \ref{sec:cross-polytope-program}, has $\Delta\approx0.824858$. Non-trivial upper bounds on $\Delta_n(K_n)$ in small dimensions $n\geq 7$ were derived by Fejes T\'oth, Fodor and V\'igh \cite{toth2015packing} using a method of Blichfeldt.

\begin{table}[]
    \centering
    \small
    {
    \begin{tabular}{|c|c|c|c|c|c|}
    \hline
        $n$ & Generator matrix & $\lambda_1$ & $\Delta$ & $\kappa_1$ & Source\\[0.7ex] 
         \hline
         $1$ & Identity & $1$ & $1$ & $2$ & Trivial \\[0.7ex] 
         \hline
         $2$ & Hadamard matrix &$\sqrt{2}$ & $1$ & $8$ & Trivial \\[0.7ex] 
         \hline
         $3$ & $\text{circ}(3,-2,1)$ & $1.78467$ & $0.947368$ & $14$ & Minkowski \cite{minkowski} \\[0.7ex] 
         \hline
         $4$ & $\begin{pmatrix}
             20 & 53 & -53 & -42 \\
             -62 & -22 & 42 & -42 \\
             -22 & -22 & -62 & 62 \\
             20 & 53 & 42 & 53
         \end{pmatrix}$ & $2.10934$ & $0.824858$ & $30$ & Nonlinear optimization \\[0.7ex] 
         \hline
         $5$ & $\text{circ}(-1,-4,4,5,6)$ & $2.41383$ & $0.682885$ & $50$ & Cools and Govaert \cite{cools2003five}\\[0.7ex] 
         \hline
         $6$ & $\text{circ}(-3,-4,5,4,9,3)$ & $2.69874$ & $0.536574$ & $72$ & Cools and Govaert \cite{cools2003five}\\[0.7ex] 
         \hline
    \end{tabular}
    }
    \caption{Dense lattice packings of cross-polytopes in small dimensions.}
    \label{tab:lattice_cross_packings_small_dim}
\end{table}

Table \ref{tab:lattice_cross_packings_small_dim} shows the densest lattice cross-polytope packings that we are aware of in dimensions $1$--$6$.  The five- and six-dimensional lattices were found by Cools and Govaert \cite{cools2003five} by a computer search, where the search was restricted to the space of lattices generated by circulant and skew-circulant matrices. Next, we will review some tools and results from convex and discrete geometry that will become useful when introducing our method for finding dense packings of convex bodies (especially cross-polytopes) in $\R^n$.


\subsection{Tools from convex and discrete geometry}

    A good reference for the material contained in this section is the book \cite{gruber2007convex}. Let $K\subseteq\R^n$ be a convex body, symmetric about the origin. Then $K$ induces a norm $\norm{\cdot}_K:\R^n\mapsto \R_{\geq 0}$ given by $\norm{x}_K:=\inf\{t\in \R_{\geq 0}:x\in tK\}$ for all $x\in \R^n$. Conversely, if $\norm{\cdot}$ is a norm-function on $\R^n$, then $K_{\norm{}}:=\{x\in\R^n:\norm{x}\leq 1\}$ is a convex body, due to the triangle inequality. A lattice $\Lambda\subseteq\R^n$ is called admissible with respect to $K$, if $\text{int}(K)\cap\Lambda=\{0\}$, or equivalently, $\norm{x}_K\geq 1$ for all $x\in\Lambda\setminus\{0\}$. An admissible lattice $\Lambda$ with respect to $K$ is called (locally) critical, if the volume $\vol{\Lambda}$ is a (local) minimum; in other words, any (slight) perturbation of $\Lambda$ which keeps the lattice admissible increases its volume. Furthermore, $\Lambda$ is admissible with respect to $K$ if and only if $\{x+\frac{1}{2}K:x\in \Lambda\}$ forms a lattice packing of $\frac{1}{2}K$, and criticality of $\Lambda$ amounts to optimality of the corresponding packing.

    Given a critical lattice $\Lambda$ with respect to a centrally symmetric convex body $K$, we are interested in the structure of the set $\Lambda\cap \partial K$, where $\partial K=K\setminus\text{int}(K)$ denotes the boundary of $K$. Define the \emph{minimal vectors} of $\Lambda$ with respect to a given norm $\norm{\cdot}_K$ induced by $K$ as the set $\mathcal{M}_K(\Lambda):=\{x\in \Lambda\setminus\{0\}:\norm{x}_K=1\}$. If $\mathcal{M}_K(\Lambda)$ contains $n$ $\R$-linearly independent vectors, then we say that the lattice is \emph{well-rounded} with respect to the norm $\norm{\cdot}_K$. If it further holds that $\Lambda=\text{span}_{\Z}(\mathcal{M}_K(\Lambda))$ we say that the lattice is \emph{strongly well-rounded}. An even stronger condition is that $\Lambda$ has a basis consisting of vectors of $\mathcal{M}_K(\Lambda)$; in this case, we say that $\Lambda$ possesses a basis of minimal vectors. It is a well-known fact that these three notions of well-roundedness are different already when $K$ is the familiar $\ell^2$-ball $\{x\in \R^n:\norm{x}_2\leq 1\}$. The kissing number $\kappa_K(\Lambda)$ of $\Lambda$ with respect to the norm $\norm{\cdot}_K$ is the cardinality of $\mathcal{M}_K(\Lambda)$. One might wonder whether locally critical lattices are always well-rounded. This is always the case, as the next proposition shows. 

\begin{prop}
\label{prop:well-rounded_convexbody}
    Let $K\subseteq\R^n$ be a convex body, symmetric about the origin. Let $\Lambda$ be a locally critical lattice with respect to $K$. Then $\Lambda$ is well-rounded with respect to $\norm{\cdot}_K$.
\end{prop}

\begin{proof}
     Suppose that $V:=\text{span}_{\R} \mathcal{M}_K(\Lambda)$ is a proper vector subspace of $\R^n$. Write $\R^n=V\oplus V^\perp$ where $V^\perp=\{x\in \R^n:x\cdot y=0 \text{ for all $y\in V$}\}$. Consider the linear transformation $\mathcal{T}: V\oplus V^\perp \to V\oplus V^\perp$, $(v,v')\mapsto (v,(1-\varepsilon)v')$ where $\varepsilon>0$ is a sufficiently small real number such that $\norm{\mathcal{T}(x)}_K\geq 1$ holds for all $x\in \Lambda\setminus\{0\}$. Such a transformation exists, since $\mathcal{T}$ acts as the identity on $\mathcal{M}_K(\Lambda)$, and there exists a $\delta>0$ such that $\norm{x}_K\geq 1+\delta$ for all nonzero $x\in \Lambda\setminus \mathcal{M}_K(\Lambda)$, since $\Lambda$ is a discrete subgroup. It remains to show that the lattice $\mathcal{T}(\Lambda)$ has smaller volume than $\Lambda$. This follows from the fact that the matrix representation of $\mathcal{T}$ with respect to a basis $B$ of $V\oplus V^\perp$ is $[\mathcal{T}]_B=\begin{pmatrix}
          I_\ell & 0 \\
          0 & (1-\varepsilon)I_m
     \end{pmatrix}$
     where $\ell=\dim V$ and $m=n-\ell$. This matrix has determinant $\det([\mathcal{T}]_B)=(1-\varepsilon)^m<1$, so $\vol{\mathcal{T}(\Lambda)}=(1-\varepsilon)^m\vol{\Lambda}<\vol{\Lambda}$, which contradicts the fact that $\Lambda$ is locally critical.
\end{proof}

In fact, one can say more about the cardinality of $\mathcal{M}_K(\Lambda)$. The following result by Swinnerton-Dyer \cite{swinnerton1953extremal} gives a lower bound on the number of minimal vectors in a locally optimal lattice packing of any convex body.

\begin{prop}
    \label{prop:swinnerton-dyer_kissing_bound}
    Suppose that $K\subseteq \R^n$ is a convex body, symmetric about the origin. If $\Lambda\subseteq\R^n$ is a locally critical lattice with respect to $K$, then $\kappa_K(\Lambda)\geq n(n+1)$.  
\end{prop}

In the other direction, a result by Minkowski \cite{minkowski1957diophantische} states that for any convex body $K$ and admissible lattice $\Lambda$, $\kappa_K(\Lambda)\leq 3^{n}-1$. If $K$ is strictly convex, then furthermore $\kappa_K(\Lambda)\leq 2^{n+1}-2$.

\subsection{A program for finding locally optimal lattice packings of the cross-polytope}

\label{sec:cross-polytope-program}

Minkowski's original proof of the optimality of the three-dimensional lattice cross-polytope packing shown in Table \ref{tab:lattice_cross_packings_small_dim} uses information about which integral linear combinations of basis vectors of a given critical lattice must lie on the boundary of the cross-polytope. Minkowski's method works more generally for convex bodies in three dimensions. In \cite{betke2000densest}, based on the work of Minkowski, Betke and Henk presented an algorithm for computing the densest lattice packing of any polytope in dimension 3. However, in dimensions greater than or equal to four, there is no known classification of configurations of minimal vectors in critical lattices. If $B=\{b_1,\dots, b_n\}$ is a basis for the lattice $\Lambda\subseteq \R^n$ which has minimum $1$ with respect to the norm induced by   $K\subseteq \R^n$, then by minimal configuration (with respect to this basis and norm) we mean a set $M\subseteq \Z^n$ for which $m=(m_1,\dots,m_n)\in M$ if and only if $\norm{\sum_{i=1}^n m_i b_i}_K=1$.

Let $K^*=\{y\in\R^n: x\cdot y\leq 1 \text{ for all $x\in K$}\}$ denote the \emph{polar body} of $K$. For instance, the polar body of the cross-polytope is $K_n^*=\{x\in \R^n: |x_i|\leq 1 \text{ for all $i=1,\dots,n$}\}$, the cube with side-length 2, centered at the origin (the unit ball with respect to the max-norm). The following lemma restricts the size of the coefficients of a vector $v\in \R^n$, when expressed as a linear combination of a basis of $\R^n$, in terms of the polar norm of the dual basis vectors and the length of $v$ in the original norm.

\begin{lem}
    \label{lem:polar_bound}
    Let $\{b_1,\dots, b_n\}$ be a basis of $\R^n$, with dual basis $\{b_1^*,\dots, b_n^*\}$ and suppose that $\norm{\sum_{i=1}^n c_i b_i}_K\leq \ell$ for some real numbers $c_i$, $i=1,\dots, n$ and positive constant $\ell$. Then $|c_i|\leq \norm{b_i^*}_{K^*} \ell$ for all $i=1,\dots, n$.
\end{lem}
\begin{proof}
    Note that $c_i=b_i^*\cdot (\sum_{j=1}^n c_j b_j)$ and apply the inequality $|x\cdot y|\leq \norm{x}_K\norm{y}_{K^*}$ ($x,y\in \R^n$) which is elementary to verify using the definition of polar body.
\end{proof}

 We are interested in the case $K=K_n$, in which case the polar norm is the max-norm $\norm{x}_\infty=\max_{1\leq i\leq n}|x_i|$. The next lemma gives a first bound on the max-norm of dual basis vectors, when the basis vectors have unit $\ell^1$-norm.

\begin{lem}
    \label{lem:max_bound}
    Let $\{b_1,\dots, b_n\}$ be a basis of $\R^n$ such that $\norm{b_i}_1=1$ for all $i=1,\dots,n$ and let $\{b_1^*,\dots, b_n^*\}$ be the dual basis. Then $$\norm{b_i^*}_\infty\leq \frac{1}{|\det(B)|}\qquad\text{for all $i=1,\dots,n$}$$
    where $B$ is the matrix with rows $b_1,\dots,b_n$.
\end{lem}

\begin{proof}
    Write $b_i^*=\frac{1}{\det(B)}\sum_{j=1}^n (-1)^{i+j} \det(B_{ij}) e_j$, where $B_{ij}$ denotes the $(n-1)\times (n-1)$ matrix obtained by deleting row $i$ and column $j$ in $B$. Since the rows of $B_{ij}$ have $\ell^1$-norm upper bounded by 1, Hadamard's inequality gives
    $$|\det(B_{ij})|\leq \prod_{k=1}^{n-1}\norm{(B_{ij})_k}_2 \leq \prod_{k=1}^{n-1}\norm{(B_{ij})_k}_1\leq 1,$$
    where $(B_{ij})_k$ denotes the $k$:th row of $B_{ij}$. The claim follows.
\end{proof}

Equality holds in Lemma \ref{lem:max_bound} if $b_i$ are the standard basis vectors. A consequence of the above lemma is the following restriction on minimal configurations in lattices with minimal $\ell^1$-norm bases. 

\begin{cor}
    \label{cor:max_bound_critical}
    If $\Lambda\subseteq\R^n$ is a lattice possessing a basis $\{b_1,\dots, b_n\}$ of minimal vectors with respect to the $\ell^1$-norm, then $m=\sum_{i=1}^n m_i b_i\in \mathcal{M}_{K_n}(\Lambda)$ implies that $|m_i|\leq n!$.
\end{cor}
\begin{proof}
    This follows from Lemmas \ref{lem:polar_bound} and \ref{lem:max_bound}, as well as the trivial bound $\Delta= \frac{1}{n! \text{vol}(\Lambda)}\leq 1$ on the packing density of the cross-polytope.
\end{proof}
We remark that the equivalence of norms on $\R^n$ allows one to derive similar bounds for arbitrary norms. Now, the bound $n!$ is not tight, since lattices with good cross-polytope packing density (small determinant) do not have minimal bases which are close to the standard basis vectors. However, the next lemma shows that locally critical lattices have many minimal vectors with small max-norm. Thus, if a lattice has a minimal basis consisting of such vectors, we get a better restriction than what Corollary \ref{cor:max_bound_critical} gives.

\begin{lem}
    \label{lem:max_norm_pidgeonhole_principle}
    Let $\Lambda$ be a locally optimal lattice packing of $K_n$. Then $\Lambda$ has at least $\frac{n(n-1)}{2}$ pairs of minimal vectors $v\in \mathcal{M}_{K_n}(\Lambda)$ satisfying $\norm{v}_\infty\leq \frac{1}{2}$.
\end{lem}

\begin{proof}
    Since $\Lambda$ is locally optimal, Proposition \ref{prop:swinnerton-dyer_kissing_bound} gives that there are at least $\frac{n(n+1)}{2}$ $\pm$-pairs of minimal vectors in $\Lambda$.
    Suppose that there are $>n$ pairs of minimal vectors in $\Lambda$ each of which has max-norm $>\frac{1}{2}$. Then by the pigeonhole principle, there exists $i\in\{1,2,\dots, n\}$ such that at least two of them, say $x,y\in\mathcal{M}_{K_n}(\Lambda)$  have $|x_i|,|y_i|>\frac{1}{2}$. Replacing $x,y$ by their negatives and swapping $x$ and $y$ if necessary, we may assume that $x_i\geq y_i>\frac{1}{2}$. But then
    \begin{align*}
        \norm{x-y}_1&=\sum_{j=1}^n |x_j-y_j|\leq |x_i-y_i|+\sum_{j\neq i} (|x_j|+|y_j|) \\
        &= |x_i-y_i|+2-|x_i|-|y_i|\\
        &= 2-2y_i<1,
    \end{align*}
    a contradiction.
\end{proof}
Using the previous lemma, we can strengthen Corollary \ref{cor:max_bound_critical}, assuming the lattice has a minimal basis consisting of vectors with max-norm upper bounded by $\frac{1}{2}$.

\begin{lem}
    \label{lem:dim4_max_bound}
    Let $\Lambda\subseteq\R^n$ be a lattice possessing a basis $\{b_1,\dots,b_n\}$ of minimal vectors with $\norm{b_i}_\infty\leq \frac{1}{2}$ for all $i=1,\dots,n$. Then $m=\sum_{i=1}^n m_i b_i\in\mathcal{M}_{K_n}(\Lambda)$ implies that $|m_i|\leq  2^{\frac{1-n}{2}}n!$.
\end{lem}


\begin{proof}
    Let $B$ be the matrix with rows $b_1,\dots,b_n$. As in Lemma \ref{lem:max_bound}, we want to find an upper bound for $|\det(B_{ij})|$, where $B_{ij}$ is as in Lemma \ref{lem:max_bound}. Since the entries of $B_{ij}$ have absolute value upper bounded by $\frac{1}{2}$, and the rows have $\ell^1$-norm upper bounded by 1,
    \begin{align*}
        |\det(B_{ij})|\leq \prod_{k=1}^{n-1} \norm{(B_{ij})_k}_2\leq\left(\sqrt{2\cdot\left(\frac{1}{2}\right)^2}\right)^{n-1}=\frac{1}{2^{\frac{n-1}{2}}}.
    \end{align*}
    Therefore, $\norm{b_i^*}_\infty <\frac{|\det(B_{ij})|}{|\det(B)|}\leq \frac{n!}{2^{\frac{n-1}{2}}}$ and the claim follows from Lemma \ref{lem:polar_bound}. 
\end{proof}

Many of the minimal configurations allowed by Corollary \ref{cor:max_bound_critical} and Lemma \ref{lem:dim4_max_bound} are not possible: for example, $(3,1,0,\dots,0)\not\in M$ since $\norm{3b_1+b_2}_1\geq |3\norm{b_1}_1-\norm{b_2}_1|=2$, and $\{(2,1,0,\dots,0),(1,-1,0,\dots,0)\}\not\subseteq M$, since otherwise $3=\norm{3b_1}_1=\norm{(2b_1+b_2)+(b_1-b_2)}_1\leq 2$. Let us next formulate the problem of determining the densest lattice cross-polytope packing in a given dimension among lattices possessing a basis of shortest vectors as a set of nonlinear optimization problems. Let $\mathcal{M}_n$ denote a finite set containing all the possible minimal configurations of a critical lattice in $\R^n$: we require that $e_1,\dots, e_n\in M$ and that $|M|\geq \frac{n(n+1)}{2}$ for every $M\in\mathcal{M}_n$ (so only one of the pairs $\pm m$ is in $M$). Further, the sets $M\in\mathcal{M}_n$ should be consistent in the sense that the triangle inequality should not be violated, nor should the coefficient bounds from Corollary \ref{cor:max_bound_critical} (or Lemma \ref{lem:dim4_max_bound} if the additionally condition $\norm{b_i}_\infty\leq \frac{1}{2}$ is imposed) be violated.

For each $M\in \mathcal{M}_n$, we define a finite set $S_M$, which is a set for which $\norm{sB}_1\geq 1$ for all $s\in S_M$ implies $\norm{x B}_1\geq 1$ for all $x\in \Z^n\setminus\{0\}$, if $B$ is the generator matrix of a lattice with minimal configuration containing $M$. Note that $S_M$ is finite since we can always take it to be the origin-centered cube with side-length $n!$. Then the algorithm is as follows: For each $M\in \mathcal{M}_n$, solve the nonlinear optimization problem in $n^2$ variables:
\begin{align*}
    \min &\det(B)^2 \\
     \text{subject to }\norm{xB}_1&=1 \text{ for all $x\in M$}\\
     \norm{xB}_1&\geq 1 \text{ for all $x\in S_M\setminus M$. }
\end{align*}

We remark that \cite{dostert2020new} used a similar method to find locally dense packings of the $p$-ball in dimension 3. The above is a general method; for it to be applicable in practice, the number of minimal configurations should be reduced significantly, and the dimension should be quite small. This procedure, with the constraint $0\leq |B_{ij}|\leq \frac{1}{2}$, was applied in dimension four to a small subset of $\mathcal{M}_4$ consisting of elements with small coefficients. We were able to find the four-dimensional lattice shown in Table \ref{tab:lattice_cross_packings_small_dim}.

\section{Simulations}
\label{sec:simulations}

 The simulations model a SISO (single-input single-output) \linebreak Rayleigh fading wiretap channel with lattice coset coding, as described in Section \ref{sec:channel_model}. The simulations were run on the program \emph{Planewalker}, implemented by Pyrrö et al. \cite{pyrro2018planewalker}. Simulations were conducted in dimensions $n=3$ and $n=4$, with the signal constellation being $4$-PAM and $16$-PAM, respectively. The number of simulation runs was $N=3\cdot 10^6$. The first set of simulations are shown in Figures \ref{fig:dim3_l1_norm} and \ref{fig:dim4_l1_norm}. They use the same lattice as superlattice and as sublattice, with different scalings. The superlattice is scaled to unit volume, and the codebook size is $8$ and $256$, respectively.

For dimension 3, we use the orthogonal lattice $\Z^3$, the dual $L_3^\ast$ of Minkowski's \cite{minkowski} densest lattice cross-polytope packing in dimension $3$ (see Table \ref{tab:lattice_cross_packings_small_dim}) and $D_3=\{x\in\Z^3:\sum_{i=1}^3 x_i\equiv 0\pmod{2}\}$. The lattices used in dimension 4 are the orthogonal lattice $\Z^4$, a Hadamard rotation $H_4$ of $\Z^4$, $K_4$, the optimal rotation of $\Z^4$ minimizing product distance (from the table in \cite{viterbo2005table}), $L_4^*$ the dual of the densest known lattice cross-polytope packing in dimension 4 (see Table \ref{tab:lattice_cross_packings_small_dim}), and $D_4=\{x\in\Z^4:\sum_{i=1}^4 x_i\equiv 0\pmod{2}\}$. 

We further repeated the simulations in dimension 4 with a fixed superlattice: the $K_4$ lattice described earlier. This lattice was chosen, since it has a maximal minimum product distance among rotations of the $\Z^4$ lattice, hence giving a high correct decoding probability for large SNRs (for justifications, see the last paragraph of Section \ref{sec:channel_model}). The statistics of the sublattices is summarized in Table \ref{tab:simulation2data}. The lattices were found by sampling a random $4\times 4$ integer matrix with entries ranging from $-9$ to $9$, with absolute determinant equal to $256$, whence the lattice generated by this matrix is an index 256 sublattice of $\Z^4$. An exception is $\Lambda_2$, which was selected from the paper \cite{damir2021WR}. This lattice is an index 256 well-rounded (with respect to the $\ell^2$-norm) sublattice of $\Z^4$. In \cite{damir2021WR}, it is argued that $\ell^2$-norm well-rounded lattices are good candidates for lattice coset codes in wiretap channels, since the problem of minimizing the product distance can be restricted to the set of well-rounded lattices without loss of generality, and since the dual lattice of a global minimizer of the theta series is well-rounded. 
The generator matrix for each sublattice $\Lambda_i$ was rotated by $K_4$ to make it a sublattice of $K_4$. Thereafter, the $\ell^1$-minimum of the dual lattice $(\Lambda_i K_4)^*$ was computed. We selected lattices with varying dual minima to test the validity of the design criterion of Corollary \ref{cor:probability_bounds} (using the fact that the $\ell^1$-theta function is dictated by the cross-polytope packing density for small $q$). The simulation results are presented in Figure \ref{fig:sim2}.

\begin{table}[H]
    \centering
    \small
    {\renewcommand{\arraystretch}{1.4}%
    \begin{tabular}{|c|c|c|}
    \hline
        $\Lambda$ & Basis vectors &  $\lambda_1((\Lambda_iK_4)^\ast)$   \\ [0.7ex] 
        \hline
        $\Lambda_{1}$ & $(-2,5,-4,-4)$, $(2,2,-2,4)$, $(2,4,-5,-3)$, $(0,-4,1,1)$  & 1.91095  \\ [0.7ex] 
         \hline
         $\Lambda_{2}$ & $(0,2,0,-4)$, $(4,0,0,2)$, $(0,0,-4,-2)$, $(3,-3,-1,-1)$  & 1.63028  \\ [0.7ex] 
         \hline
         $\Lambda_3$ & $(0,5,0,4)$, $(-1,1,1,-2)$, $(-1,0,0,6)$, $(8,8,2,-8)$ & $1.38677$ \\[0.7ex] 
\hline
$\Lambda_{4}$ & $(3,-1,-6,-4)$, $(2,-6,-4,-2)$, $(-8,-7,-3,-4)$, $(1,-7,-6,-3)$  & 1.1595 \\ [0.7ex] 
         \hline

         $\Lambda_{5}$ & $(6,9,-6,8)$, $(1,8,1,-4)$, $(4,-2,-2,1)$, $(7,-6,-7,9)$  & 0.990213 \\ [0.7ex] 
         \hline

         $\Lambda_{6}$ & $(-5,3,4,7)$, $(1,7,4,5)$, $(6,4,2,-3)$, $(9,2,-4,-9)$  & 0.705421 \\ [0.7ex] 
         \hline
    \end{tabular}}
    \caption{Statistics for lattices used in simulation in dimension 4 with fixed superlattice.}
    \label{tab:simulation2data}
\end{table}

\begin{figure}[H]
    \centering
    \includegraphics[width=13cm]{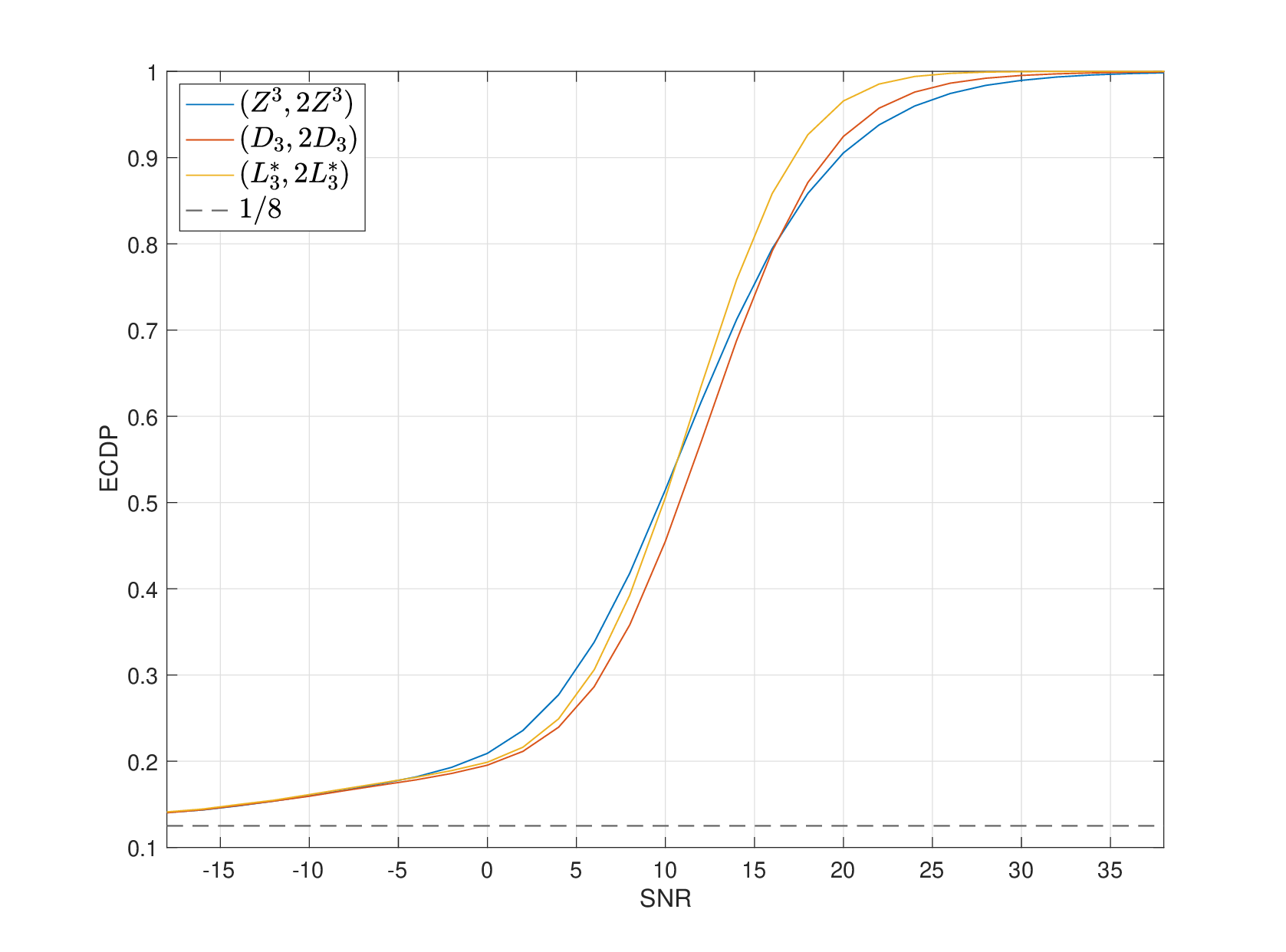}
    \caption{Simulation results in dimension 3.}
    \label{fig:dim3_l1_norm}
\end{figure}


\begin{figure}[H]
    \centering
    \includegraphics[width=12.6cm]{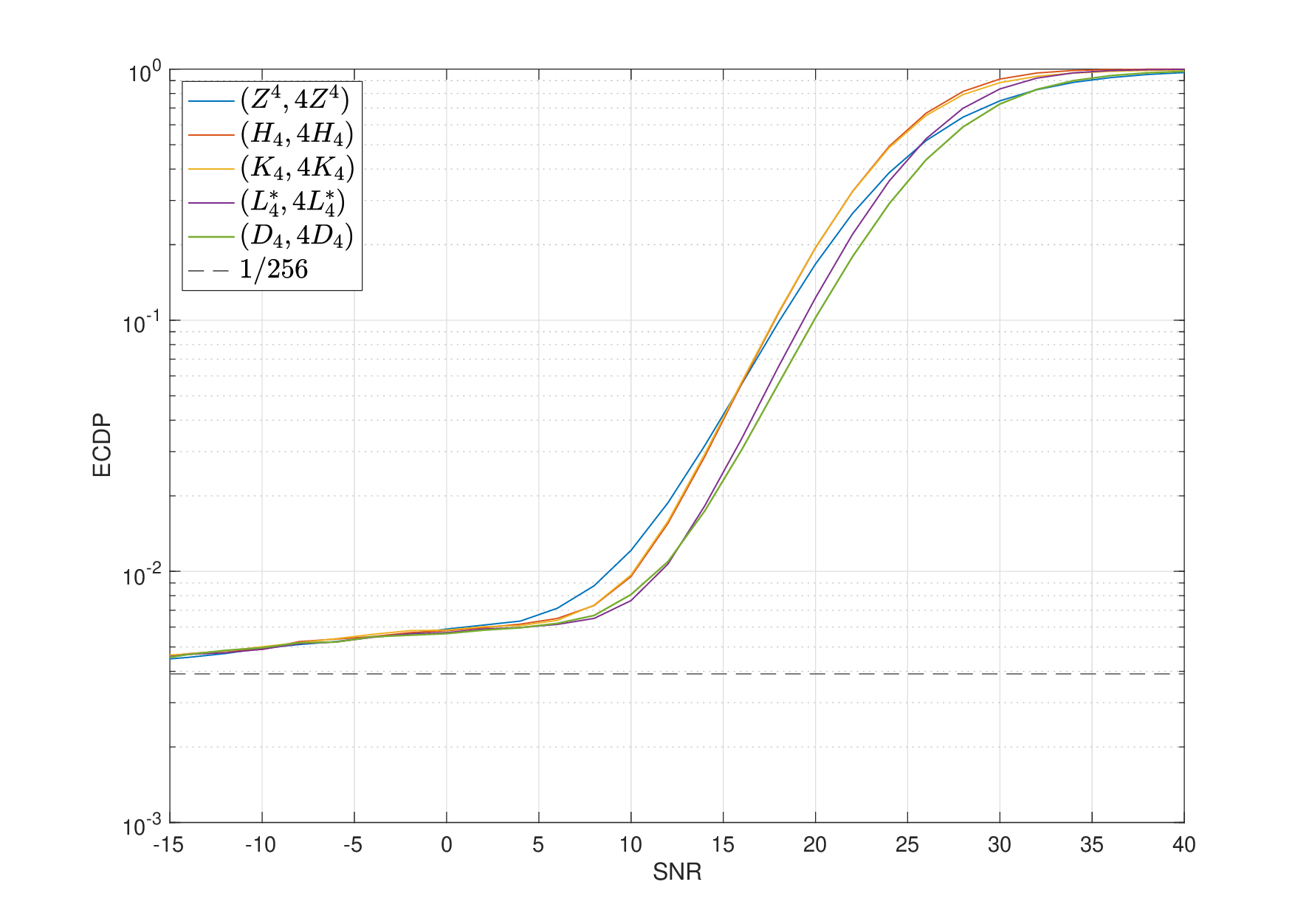}
    \caption{Simulation results in dimension 4.}
    \label{fig:dim4_l1_norm}
\end{figure}

\begin{figure}[H]
    \centering
    \includegraphics[width=12.6cm]{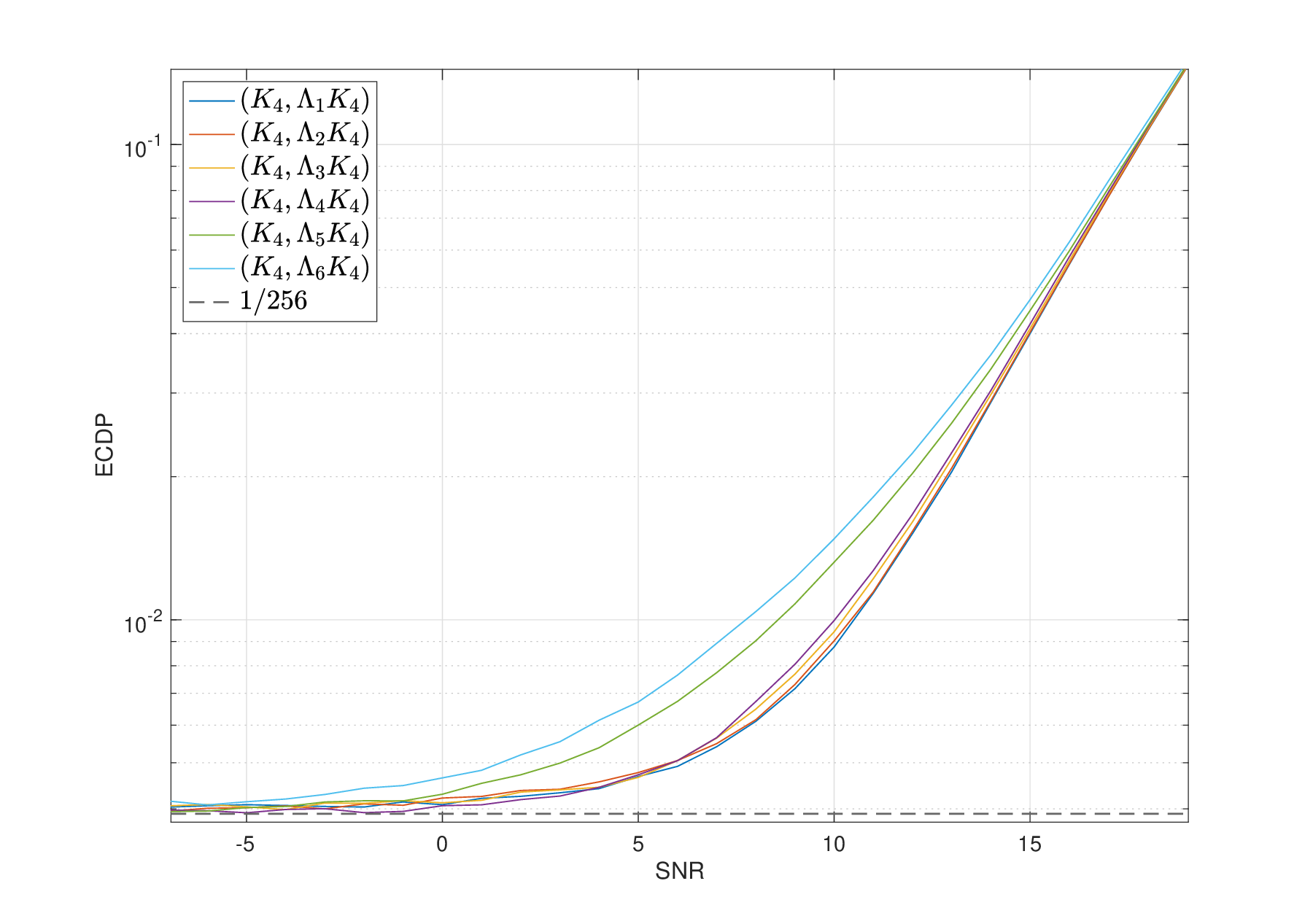}
    \caption{Simulation results in dimension 4 with $K_4$ as superlattice.}
    \label{fig:sim2}
\end{figure}\pagebreak

The simulations suggest that the $\ell^1$-minimum of the sublattice dual plays a crucial role in the performance of a given code, in the small SNR range. This is perhaps best illustrated in Figure \ref{fig:sim2}: in the SNR range $6$--$15$ dB, the lattices perform according to their dual $\ell^1$-minima. This phenomenon is also illustrated in Figure \ref{fig:dim4_l1_norm}, where the lattice pair $(L_4^*,4L_4^*)$ performs best out of the selected lattices at $\text{SNR}=10\text{ dB}$, while the pair $(\Z^4,4\Z^4)$ performs worst. For small enough SNR-values, the trend is no longer very clear. For larger SNR-values, the codes whose superlattices have large diversity and product distance (or small dual theta function, by part 1 of Corollary \ref{cor:probability_bounds}), seem to perform better. 


\section*{Acknowledgements}

We would like to thank Camilla Hollanti and Pavlo Yatsyna for fruitful discussions and comments.

\bibliographystyle{siamplain}
\bibliography{refs}

\end{document}